\documentclass[12pt,letterpaper,titlepage,reqno]{amsart}
\usepackage{amsmath, amssymb, amsthm, amsfonts,amscd,amsaddr,wasysym,enumerate}
\usepackage[
paper=a4paper,
portrait=true,
textwidth=425pt,
textheight=650pt,
tmargin=3cm,
marginratio=1:1
]{geometry}

\theoremstyle{plain}
\newtheorem{theorem}{Theorem}[section]

\newtheorem{cor}[theorem]{Corollary}

\newtheorem{prop}[theorem]{Proposition}
\newtheorem{lemma}[theorem]{Lemma}

\theoremstyle{definition}
\newtheorem{definition}[theorem]{Definition}

\numberwithin{equation}{section}
\newtheorem*{theoremA*}{Theorem A}
\newtheorem*{theoremB*}{Theorem B}
\newtheorem*{theorem1*}{Theorem A'}
\newtheorem*{theoremC*}{Theorem C}
\newtheorem*{theoremD*}{Theorem D}
\newtheorem*{theoremE*}{Theorem E}
\newtheorem*{theoremF*}{Theorem F}
\newtheorem*{theoremE2*}{Theorem E2}
\newtheorem*{theoremE3*}{Theorem E3}

\newcommand{\C}{\mathbb{C}}

\newcommand{\Hc}{\mathcal{H}}
\newcommand{\cH}{\mathcal{H}}

\newcommand{\R}{\mathbb{R}}
\newcommand{\N}{\mathbb{N}}

\newcommand{\Aut}{\operatorname{Aut}}

\newcommand{\SO}{\operatorname{SO}}

\newcommand{\End}{\operatorname{End}}

\newcommand{\GL}{\operatorname{GL}}

\newcommand{\Lie}{\operatorname{Lie}}

\newcommand{\diag}{\operatorname{diag}}

\def\hat{\widehat}
\def\af{\mathfrak{a}}

\def\e{\epsilon}
\def\gf{\mathfrak{g}}

\def\1{{\bf1}}

\newcommand{\even}{\operatorname{even}}
\newcommand{\odd}{\operatorname{odd}}
\newcommand{\Ec}{\mathcal{E}}

\newcommand{\sslash}{\mathbin{/\mkern-6mu/}}
\usepackage[usenames]{color}


\begin{document}
    \title{A note on $L^p$-factorizations of representations}

	\date{\today}
	
	\begin{abstract} In this paper we give an overview on $L^p$-factorizations of Lie group representations and introduce the notion of smooth $L^p$-factorization. 		
	\end{abstract}
	\author[Ganguly]{Pritam Ganguly}
	\email{pritam1995.pg@gmail.com}
	\address{Universit\"at Paderborn, Institut f\"ur Mathematik\\Warburger Stra\ss e 100,
		33098 Paderborn}
	\author[Kr\"otz]{Bernhard Kr\"{o}tz}
	\email{bkroetz@gmx.de}
	\address{Universit\"at Paderborn, Institut f\"ur Mathematik\\Warburger Stra\ss e 100,
		33098 Paderborn}
	
	\author[Kuit]{Job J. Kuit}
	\email{j.j.kuit@gmail.com}
	\address{Universit\"at Paderborn, Institut f\"ur Mathematik\\Warburger Stra\ss e 100,
		33098 Paderborn}
	
	\maketitle

\centerline{\it This article is dedicated to the fond memories of Gerrit van Dijk.}

\section{Introduction}

Gerrit van Dijk had a liking for the theory of Gel'fand pairs and Gel'fand pairs were a recurring theme in his research. He was interested in pairs $(G,H)$ where $G$ is a noncompact semisimple Lie group and $H$ a non-compact symmetric subgroup. The definition of what should be a Gel'fand pair for non-compact $H$ is debatable as there is no natural Hecke algebra of $H$-bi-invariant functions to our disposal.  The notion Gerrit worked with readily translates into the criterion that the space of $H$-invariant distribution vectors
attached to a unitary irreducible representation is at most one-dimensional.
 A space that occurs regularly in his work is the hyperboloid
$$
X_{n}
=\SO_0(n,1)/\SO_0(n-1,1)\qquad (n\geq 2).
$$
The pair $(G,H)=\big(\SO_0(n,1),\SO_0(n-1,1)\big)$ happens to be not Gel'fand.
 A beautiful result of Gerrit \cite{vanDijk} asserts that all rank one symmetric spaces have the Gel'fand property with the exception of the $X_{n}$ above.

The space $X_n=G/H$ was considered recently in \cite{KKS} where the failure of the Gel'fand property is shown directly in case $n\geq 4$ and it is shown that there are quantitative differences between the various $G$-equivariant embeddings $\pi\hookrightarrow C^{\infty}(G/H)$ for certain irreducible subrepresentations $\pi$ of $L^{2}(G/H)$.

In this article we address some questions related to the phenomenon, named after Kunze and Stein, that for semisimple Lie groups $G$ the convolution product
$$
C_{c}(G)\times L^{2}(G)\to L^{2}(G),\quad (f,h)\mapsto f*h
$$
defined by
$$
f\ast h(x)=\int_{G}f(y)h(y^{-1}x)~dy,
$$
extends to a continuous map
$$
L^p(G)\times L^2(G)\to L^2(G)
$$
for all $1\leq p<2$. See \cite{C}.
In 1960, Kunze and Stein \cite{KS} discovered this result in the case $G=SL_2(\R)$ as a byproduct of their study of  analytic continuation of  the principal series  as uniformly bounded representations.

We set out to investigate whether a Kunze-Stein phenomenon could hold for symmetric spaces $G/H$, i.e., whether the natural convolution map
$$
C_{c}(G)\times L^{2}(G/H)\to L^{2}(G/H),\quad (f,h)\mapsto f*h
$$
extends to a continuous map
$$
L^p(G)\times L^2(G/H)\to L^2(G/H)
$$
for some or all $1< p<2$.
\par The Kunze-Stein phenomenon for $G$  may be put in a slightly more general framework by saying that a representation $(\pi,E)$ of $G$ is $L^{p}$-factorizable  if the natural map
$$
C_{c}(G)\times E\to E,\quad (f,v)\mapsto \pi(f)v
$$
factors through $L^{p}(G)\times E$, i.e.,  extends to a continuous map
$$
L^p(G)\times E\to E.
$$
We introduce a new notion:  $(\pi,E)$ is smoothly-$L^{p}$-factorizable if the natural map
$$
C_{c}^{\infty}(G)\times E^{\infty}\to E^{\infty},\quad (f,v)\mapsto \pi(f)v
$$
extends to a continuous map
$$
L^p(G)^{\infty}\times E^{\infty}\to E^{\infty}.
$$
Here $L^p(G)^{\infty}$ denotes the space of smooth vectors for the (left and right) regular of $G\times G$ on $L^p(G)$.
Clearly every $L^{p}$-factorizable representation is also smoothly-$L^{p}$-factorizable but the converse is not clear. A natural question would be:
\medbreak \noindent{\bf Open question:} Is a unitary representation $L^p$-factorizable if it is smoothly $L^p$-factorizable?
 \medbreak

We consider this question worth of further investigation, in particular because smooth factorizability appears to be easier to establish. For instance we show by fully elementary means that the smooth Fr\'echet representation on the Harish-Chandra Schwartz space is smoothly-$L^p$-factorizable for all
$1\leq p<2$. See Corollary \ref{cor smoothlp}.
\par We return to the hyperboloids and the main result of this article. We show that
$L^2(X_n)$ is not smoothly-$L^p$-factorizable for
$p>\frac{n-1}{n-2}$ which implies in particular that $L^2(X_n)$ is not $L^p$-factorizable.
This means that there is no straightforward generalization of the Kunze-Stein phenomenon for the group $G=G/\{e\}$ to homogeneous spaces $X=G/H$, in particular, non-Riemannian symmetric spaces.

Upon completion of this work we came across the recent work of Samei-Wiersma \cite{SW} which combined with \cite{BK} yields the following result:

\begin{theorem}
Let $X=G/H$ where $G$ is semi-simple and $H\subseteq G$ a reductive subgroup.  Then there exists a $1<p_{X}<2$ such that $L^{2}(X)$ is $L^{p}$-factorizable for all $p<p_{X}$.
\end{theorem}

In fact, \cite[Lemma 2.8]{BK} shows that the dense linear subspace of $L^{2}(X)$ generated by all compactly supported step functions has matrix-coefficients that lie in $L^{p'}(G)$ for all $p'>p'_{X}$ for some $p'_{X}>2$. We apply a variant of \cite[Theorem 1.5]{SW}, where we replace the assumption of cyclicity by density, which is allowed by the arguments in the proof of \cite[Corollary 5.5]{SW}. Proposition \ref{Prop Factorizability of Banach reps} now implies that $L^{2}(X)$ is $L^{p}$-factorizable for all $1<p<p_{X}$ with $p_{X}$ given by $\frac{1}{p_{X}}+\frac{1}{p'_{X}}=1$.

{\em Acknowledgement:} We thank Henrik Schlichtkrull and the referee for valuable comments and suggestions.

\section{Setup and Notation}
Let $G$ be a unimodular Lie group. By a representation of $G$ we understand a pair $(\pi, E)$ consisting of a complete locally convex topological vector space $E$ and a group homomorphism $\pi: G\rightarrow \GL(E)$ such that the associated action map
$$
G\times E \rightarrow E, (g,v)\mapsto \pi(g)v
$$
is continuous. If $E$ is a Banach, Hilbert or Fr\'echet space, then we say that $(\pi, E)$ is  a Banach, Hilbert or Fr\'echet representation, respectively.

Given a representation $(\pi, E)$ of $G$, we call a vector $v\in E$ smooth provided the orbit map
$$
G\rightarrow E,\quad g\mapsto \pi(g)v
$$ is smooth. We write $E^{\infty}$ for the space of all smooth vectors, which forms a $G$-invariant subspace of $E$.  We denote the restricted action of $G$ on $E^\infty$ by $\pi^{\infty}$ and equip $E^{\infty}$ with the usual topology so that $(\pi^{\infty},E^{\infty})$ forms a representation of $G$. Let $E'$ denote the continuous dual of $E$, which we equip with the usual dual norm.  To each $\lambda\in E'$ and $v\in E$, we associate the matrix coefficient
$$
m_{\lambda, v}: G\rightarrow \mathbb{C},\quad g\mapsto \lambda(\pi(g)v).
$$
We note that for a given $\lambda$ the matrix coefficient $m_{\lambda, v}$ is continuous for every $v\in E$ and smooth for every $v\in E^{\infty}$. By abuse of notation, we denote the associated algebra representation
$$
C_c(G)\times E\to E,\quad (f,v)\mapsto \int_{G}f(g)\pi(g)v\,dg
$$
by $\pi$ as well.
Note that depending on the kind of space $E$ we are working with, $\pi$ may extend to a larger algebra. For example, it is not difficult to see that $\pi$ extends to a bounded representation of  $L^1(G)$ in case $(\pi, E)$ is a bounded Banach representation, i.e., $(\pi,E)$ is a Banach representation so that  there exists a uniform bound on  the operator norms of the operators in $\pi(G)$.

Let $\mathfrak{g}:=Lie(G)$ be the Lie algebra of $G.$ The representation $\pi$ of $G$ gives rise to a representation of $\mathfrak{g}:$
$$d\pi:\mathfrak{g}\rightarrow \End(E^{\infty})$$
defined by the following prescription:
$$d\pi(X)v=\lim_{t\rightarrow 0}\frac{\pi(\exp(tX))v-v}{t} \qquad \left(X\in \mathfrak{g},  v\in E^{\infty}\right).$$
Let $(\pi, E)$ be a Banach representation of G. Let $p$ be the norm on $E$ and fix a basis $X_1,...,X_n$
of $\mathfrak{g}$. For every $k\in \mathbb{N}\cup\{0\}$, we define
$$p_k(v)=\sum_{\alpha\in(\mathbb{N}\cup\{0\})^n}p\left(d\pi(X_1^{\alpha_1} X_2^{\alpha_2}...X_n^{\alpha_n})v\right)\qquad \left(v\in E^{\infty}\right).$$ We call $p_k$, the $k$th Sobolev norm associated to the norm $p$. This family of semi-norms $\{p_k:k\in\mathbb{N}\cup\{0\}\}$ defines a Fr\'{e}chet topology on $E^{\infty}.$

\section{$L^{p}$-factorization}
In this section $G$ is a unimodular Lie group, unless otherwise stated.

\subsection{Definition and first results}
\begin{definition}
We say that a representation $(\pi,E)$ of $G$ is {\em{${L^p}$-factorizable}} provided that the canonical map
$$
C_c(G)\times E\rightarrow E,\quad(f,v)\mapsto \pi(f)v
$$
extends to a continuous bi-linear map
$$
L^p(G)\times E \rightarrow E.
$$
\end{definition}

\begin{prop}\label{Prop Factorizability of Banach reps}
Let $1\le p<\infty$ and let $1<p'\le\infty$ satisfy $\frac1p+\frac1{p'}=1$.
Let further $(\pi,E)$ be a Banach representation of $G$.
The following are equivalent.
\begin{enumerate}[(i)]
\item\label{i} $(\pi,E)$ is $L^p$-factorizable.
\item\label{ii}  $m_{\lambda,v}\in L^{p'}(G)$ for every $\lambda\in E'$ and $v\in E$.
\item\label{iii}  There exists a constant $C>0$ such that
$$\|m_{\lambda,v}\|_{p'}\leq C\|\lambda\|_{E'}\,\|v\|_E$$
for all $\lambda\in E'$ and $v\in E$.
\end{enumerate}
\end{prop}

\begin{proof} The proof is an adaption of the unitary case treated in \cite[Lemma 27]{KS}.
It follows from the Hahn-Banach theorem that the statement of (\ref{i}) is equivalent to the
existence of a constant $C>0$ such that
$$|\lambda(\pi(f)v)|\leq C\|f\|_p\|\lambda\|_{E'}\|v\|_E$$
for all $f\in C_c(G)$, $\lambda\in E'$ and $v\in E$. Since
$$\lambda(\pi(f)v)=\int_G f(g)m_{\lambda,v}(g)\,dg$$
the equivalence between (\ref{i}) and (\ref{iii}) then follows from H{ö}lder's inequality and the duality between
$L^p(G)$ and $L^{p'}(G)$.

The implication of (\ref{ii}) from (\ref{iii}) is trivial. The statement in
(\ref{iii}) amounts to joint continuity of the bilinear map $(\lambda,v)\mapsto m_{\lambda,v}$ into $L^{p'}(G)$.
It is an easy
consequence of the Banach-Steinhaus theorem that separate continuity implies joint continuity
for a bilinear map defined on Banach spaces .
Hence it remains to show that (\ref{ii}) implies
separate continuity of $(\lambda,v)\mapsto m_{\lambda,v}$ into $L^{p'}(G)$.

We assume (\ref{ii}) and consider for a fixed $v\in E$ the map
$$T_v: E'\to L^{p'}(G), \quad \lambda\mapsto m_{\lambda,v}.$$
To prove $T_v$ is bounded we use the closed graph theorem. Let $\lambda_n\to\lambda\in E'$
for $n\to\infty$, and assume $T_v\lambda_n\to m$ in $L^{p'}(G)$. Since $m_{\lambda,v}(g)$ depends
continuously on $\lambda$ for all $g\in G$, the sequence of functions $T_v\lambda_n\in C(G)$ converges pointwise
to $T_v\lambda$. Hence $T_v\lambda=m$ a.e., and $T_v\lambda_n\to T_v\lambda$ in $L^{p'}$. Hence $T_v$ is bounded.

By a similar argument, for $\lambda\in E'$ fixed  the map
$$S_\lambda: E\to L^{p'}(G), \quad v\mapsto m_{\lambda,v}$$
is bounded. Hence we have established the separate continuity, and (\ref{iii}) follows.
\end{proof}

\begin{cor}
Let $1\leq p<\infty$ and let $(\pi, E)$ be a  bounded representation on a Banach space. If $(\pi, E)$ is $L^p$-factorizable then $(\pi, E)$ is $L^q$-factorizable for all $1\leq q\leq p$.
\end{cor}

\begin{proof}
Matrix coefficients of a bounded Banach representations are bounded and $L^{p'}(G)\cap L^{\infty}(G)\subset L^{q'}(G)$ for all $q'\geq p'$. Therefore, the assertion follows from Proposition \ref{Prop Factorizability of Banach reps}.
\end{proof}

\subsection{Kunze-Stein phenomenon}
Assume for now that $G$ is a connected noncompact semisimple Lie group with finite center.
Let $L$ denote the left regular representation $G$ on $L^2(G)$ defined by the action
$$
L(g)f(x)
=f(g^{-1}x),~g\in G,~~f\in L^2(G).
$$
Given $f,h\in L^2(G)$, the associated matrix coefficient is given by
$$
m^L_{f,h}(g)
=\langle L(g)h,\overline{f}\rangle = f\ast h^{\vee}(g),
$$
where $h^{\vee}(g)=h(g^{-1})$.
Therefore, for $1<p<\infty$, by Proposition \ref{Prop Factorizability of Banach reps}, it follows that $L$ is $L^p$-factorizable if and only if there exists a $C>0$ so that
$$
\|f\ast h\|_{p'}
\leq C \|f\|_2\|h\|_2\qquad \big( f,h\in L^2(G)\big).
$$
In view of the duality, this is equivalent to
$$
\|f\ast h\|_2
\leq C \|f\|_p\|h\|_2,~~f\in L^p(G),~h\in L^2(G).
$$
In other words, $L^p$-factorizability of $L$ is equivalent to convolution extending to a continuous map
$$
	L^p(G)\times L^2(G)\to L^2(G),\quad (f,h)\mapsto f*h.
$$
If the latter holds true for all, it is called the {\em Kunze-Stein phenomenon} for $G$ and $p$. The Kunze-Stein phenomenon occurs if and only if $1\leq p<2$ (see \cite{C}). Hence the left regular representation $L$ is $L^p$-factorizable if and only if  $1\leq p<2$.

Let $\sigma$ and $\rho$ be unitary representation of $G$. We recall that $\sigma$ is
weakly contained in $\rho$ if any diagonal matrix coefficient of $\sigma$ can be approximated,
uniformly on compacta, by linear combinations of diagonal matrix coefficients of $\rho$. In view of \cite[Lemma 1.23]{Eymard}
$\sigma$ is weakly contained in $\rho$ if and only if
$$
\|\sigma(f)\|_{\sigma}\leq \|\rho(f)\|_{\rho}\qquad\big(f\in C_{c}(G)\big).
$$

\begin{cor}
If $G$ is a connected noncompact semisimple Lie group with finite center, then any unitary representation $\pi$ of $G$ which is weakly contained in $L^{2}(G)$ is $L^{p}$-factorizable for every $1\leq p<2$.
\end{cor}

\begin{proof}
Let $\lambda$ be the left-regular representation of $G$ on $L^{2}(G)$.
For every $1\leq p<2$ there exists a $C>0$ so that for every $f\in C_{c}(G)$ 
$$
\|\pi(f)\|_{\pi}\leq \|\lambda(f)\|_{\lambda}\leq C \|f\|_{p}.
$$
This implies that $\pi$ is $L^{p}$-factorizable for every $1\leq p<2$.
\end{proof}

\subsection{$L^p$-factorizability of unitary representations}
We first consider direct integrals of unitary representations.
We begin with a definition.

\begin{definition}
Let $(S, \mu)$ be a $\sigma$-finite measure space.
We say that the family $(\pi_{s}, E_{s})_{s\in S}$ of Banach representations of $G$ is {\em{${L^p}$-factorizable}} provided that that there exists a constant $C>0$ so that for almost every $s\in S$ the canonical map
$$
C_c(G)\times E_{s}\rightarrow E_{s},\quad(f,v)\mapsto \pi_{s}(f)v
$$
satisfies
\begin{equation}\label{eq Estimate pi_s(f)}
\|\pi_{s}(f)v\|_{s}\leq C\|f\|_{p}\|v\|_{s}\qquad\big(f\in C_{c}(G),v\in E_{s}\big).
\end{equation}
\end{definition}

Note that if a family $(\pi_{s}, E_{s})_{s\in S}$ is $L^{p}$-factorizable, then the representation $(\pi_{s}, E_{s})$ is $L^{p}$-factorizable for almost every $s\in S$.

\begin{prop}\label{Prop Direct integral} Let $1\leq p<\infty$, let $(S, \mu)$ be a $\sigma$-finite measure space and let
$$
\Big(\pi=\int^{\oplus}_{S}\pi_{s}\,d\mu(s), \cH=\int^{\oplus}_{S}\cH_{s}\,d\mu(s)\Big)
$$
be a direct integral of a family $(\pi_{s},\cH_{s})_{s\in S}$ of unitary representations on separable Hilbert spaces.
Then $(\pi,\cH)$ is $L^p$-factorizable if and only if the family $(\pi_s, \cH_s)_{s\in S}$ is $L^p$-factorizable. 
\end{prop}

\begin{proof}
First assume that the family $(\pi_s,\cH_s)_{s\in S}$ is $L^p$-factorizable. Let $C>0$ be so that
$$
\|\pi_s(f)v\|_{s}
\leq C \|v\|_s\|f\|_p\qquad \big(f\in L^{p}(G),v\in \cH_s\big)
$$
for almost every $s\in S$.
For any $v\in \cH$, by definition we have
\begin{align*}
\|\pi(f)v\|^2 =\int_{S}\|\pi_s(f)v_s\|^2\,d\mu(s)
\leq C^{2} \|f\|_p^{2} \int_{S}\|v_s\|^2_s\,d\mu(s)
= C^{2} \|f\|_p^{2}\|v\|^{2}
\end{align*}
proving that $(\pi, \cH)$ is $L^p$ factorizable.

For the other implication, assume that $(\pi, \cH)$ is $L^p$ factorizable. Now using the separability of $L^p(G)$, choose a countable dense subset $\{h_j:j\in\mathbb{N}\}\subset L^p(G)$. So, by assumption there exists a $C>0$ so that for all $j\in \mathbb{N}$,
$$
\|\pi(h_j)v\|\leq C \|h_j\|_p\|v\|\qquad(v\in \cH),
$$
i.e.,
$$
\mathrm{ess\, sup}_{s\in S}\|\pi_s(h_j)\|_{\cH_s\rightarrow \cH_s}\leq C\|h_j\|_p\qquad (j\in \mathbb{N}).
$$
For $j\in\N$ we define
$$
\Omega_{j}
:=\{s\in S:\|\pi_s(h_j)\|_{\cH_s\rightarrow \cH_s}> C \|h_j\|_p\} .
$$
We note that $\mu(\Omega_{j})=0$.
Let $\Omega:=\cup_{j\in \mathbb{N}}\Omega_j$. Then clearly $\mu(\Omega)=0$. Moreover, for every $j\in \mathbb{N}$ and all $s\in S\setminus\Omega$
$$
\|\pi_s(h_j)u\|_{s}
\leq C \|u\|_s\|h_j\|_p\qquad(u\in \cH_s).
$$
In other words, for every $j\in \mathbb{N}$, $$\|\pi_s(h_j)\|_{\cH_s\rightarrow \cH_s}\leq C\|h_j\|_p,~~s\in S\setminus \Omega.$$ But since $\{h_j:j\in\mathbb{N}\}$ is dense in $L^p(G)$, for $s\in S\setminus \Omega$, $\pi_s$ has unique continuous extension to $L^p$ and for any $f\in L^p(G)$, we have
$$\|\pi_s(f)\|_{\cH_s\rightarrow \cH_s}\leq C\|f\|_p,~~s\in S\setminus \Omega,$$ proving the proposition.
\end{proof}

Let $\widehat{G}$ be the set of all equivalence classes of irreducible unitary representations of $G$ equipped with the usual Fell topology.
Now the disintegration theorem for unitary representations $(\pi, \cH)$ of a type I group $G$ states that for such a group there exists a Borel  measure $\mu$ on $\widehat{G}$ such that $(\pi, \cH)$ is unitarily equivalent to the direct integral representation
$$
\left(\int_{\widehat{G}}\pi_s\,d\mu(s),~\int_{\widehat{G}}\cH_s\,d\mu(s)\right).
$$
See for example \cite[Theorem 14.10.5]{W2} for more details. This, as an immediate consequence of the above proposition, yields the following result.

\begin{cor}
\label{typeI}
Let $G$ be a type I Lie group. A unitary representation of $G$ is $L^p$-factorizable if and only if  the family of representations appearing in its disintegration is $L^p$-factorizable.
\end{cor}

As a consequence of this result, in the following we witness that there is a large sub-class of type I non-compact groups that do not admit the $L^p$-factorizability property.

\begin{cor}
Let $G$ be a type $I$ Lie group with non-compact center. Then unitary representations of $G$ are not $L^p$-factorizable for $p>1.$ In particular this holds when $G$ is non-compact and abelian, or a Carnot Lie group.
\end{cor}

\begin{proof}
In view of  Corollary \ref{typeI}, it is sufficient to prove that irreducible unitary representations of $G$ are not $L^p$ factorizable. Fix $p>1.$ Let $(\pi, H)$ be an irreducible unitary representation of $G.$  So, by the Proposition \ref{Prop Factorizability of Banach reps}, it is enough to show that there exits matrix coefficient of $(\pi, \cH)$ which is not in $L^{p'}(G).$ Now from the Schur lemma it follows that
$$
\pi(z)
=C_z I\qquad \big(z\in Z(G)\big)
$$
where $C_z$ is constant such that $|C_z|=1$ for all $z\in Z(G).$

The center $Z(G)$ is a closed subgroup of $G$ and hence the natural projection $\pi:G\to G/Z(G)$ is a principal fiber bundle. Therefore, there exits a compact neighborhood $M$ of  $e\cdot Z(G)$ in $G/Z(G)$ and a $Z(G)$-equivariant diffeomorphism $\tau:\pi^{-1}(M)\to M\times Z(G)$ so that $\pi=p\circ \tau$, where $p:M\times Z(G)\to M$ is the projection onto the first component. The pull-back of the Haar measure on $G$ along this diffeomorphism yields a product measure of a Haar measure $dz$ on $Z(G)$ and a positive Radon measure $d\nu$ on $M$. Since $Z(G)$ is non-compact, the volume of $Z(G)$ is infinite. Now for any $u\in H\setminus\{0\}$, and $(z,m)\in Z(G)\times M$, it follows that $$|\langle \pi(zm)u,u\rangle|=|\langle \pi(m)u,u\rangle|$$ which yields
$$
\int_{G}|\langle \pi(g)u,u\rangle|^{p'}dg
\geq \int_{Z(G)\times M}|\langle \pi(zm)u,u\rangle|^{p'}dzd\nu(m)=\infty.
$$
This shows  that the matrix coefficients are not in $L^{p'}(G)$, completing the proof.
\end{proof}

\subsection{Local formulation of the Kunze-Stein phenomenon}
Assume that $G$ is connected and semisimple and has finite center.
Let $P=M_{P}A_{P}N_{P}\subseteq G$ be a parabolic subgroup, $\sigma\in \hat{M}_{P}$ a discrete series representation of $M_{P}$ and $\lambda\in i\af_{P}^{*}$. The parabolically induced representation $\pi_{\sigma,\lambda}$ of $G$ is unitary. As a model Hilbert space one can take $\cH_{\sigma}:=L^{2}(K\times_{M_{P}}V_{\sigma})$ independent of $\lambda$, where $V_{\sigma}$ is a unitary model for $\sigma$. 

For the convenience of the reader we provide the following local version of the Kunze-Stein phenomenon. It is a slight strengthening of Proposition \ref{Prop Direct integral} for $\pi$ equal to the left regular repesentation of $G$ on $L^{2}(G)$.

\begin{prop}\label{Prop tempered reps}
For all $1\leq p<2$ there exists a constant $C_{p}>0$ such that for all parabolic subgroups $P=M_{P}A_{P}N_{P}$, discrete series representations $\sigma$ of $M_{P}$ and $\lambda\in i\af_{P}^{*}$
$$
\|\pi_{\sigma,\lambda}(f)\|\leq C_{p}\|f\|_{p}\qquad \big(f\in C_{c}(G)\big).
$$
\end{prop}

\begin{proof}
The Kunze-Stein phenomenon combined with Proposition \ref{Prop Direct integral} yields a constant $C_{p}>0$ such that for almost all $\pi$ in the support of the Plancherel measure for $L^{2}(G)$ we have
$$
\|\pi(f)v\|_{\pi}\leq C_{p}\|f\|_{p}\|v\|_{\pi}\qquad\big(v\in \cH_{\pi},f\in C_{c}(G) \big),
$$
which rephrases in terms of the operator norm as
$$
\|\pi(f)\|\leq C_{p}\|f\|_{p}\qquad\big(f\in C_{c}(G) \big).
$$
Now we invoke Harish-Chandra's Plancherel theorem and obtain for a fixed $\sigma$ in the discrete series of representations for $M_{P}$ and  Lebesgue almost all $\lambda\in i\af_{P}^{*}$ that
$$
\|\pi_{\sigma,\lambda}(f)v\|_{\sigma}\leq C_{p}\|f\|_{p}\|v\|_{\sigma}\qquad\big(v\in \cH_{\sigma},f\in C_{c}(G) \big).
$$
The assignment
$$
\lambda\mapsto \|\pi_{\sigma,\lambda}(f)v\|_{\sigma}
$$
is continuous. This completes the proof.
\end{proof}

\subsection{Semisimple symmetric spaces}
Assume that $G$ is a connected semisimple Lie group. Let $\sigma$ be an involution i.e., $\sigma\in \Aut(G)$, and $\sigma^2=I$. Suppose $H$ is an open subgroup of the group $G^\sigma$ of fixed points for $\sigma$. Then the pair $(G,H)$ is called a semisimple symmetric pair and the associated homogeneous space $X:=G/H$ is a semisimple symmetric space.

As a special case, when $\sigma=\theta$ is a Cartan involution and $H=K=G^{\theta}$, $X=G/K$ is a Riemannian symmetric space. Semisimple Lie groups are examples of semisimple symmetric spaces. Indeed, if $G'
$ is a connected semisimple Lie group, then take $G=G'\times G'$ and consider the action of $G$ on $G'$ given by $(g_1,g_2)\cdot x:=g_1xg_2^{-1}$. This induces the isomorphism of $G$-spaces
$$
X
:=G/H\simeq G',
$$
where $H=\diag(G')$ is the stabilizer of the identity of $G'$ in $G$. The group $H$ equals the fixed point subgroup $G^{\sigma}$ where  $\sigma$ is the involution given by $\sigma(g_1,g_2)=(g_2, g_1)$. This is referred to as the group case.

The goal of this section is to demonstrate that the Kunze-Stein phenomon does not generalize to the class of semisimple symmetric spaces. To provide a counter example we consider the $n$-dimensional one-sheeted hyperboloid $X$. The space $X$ is homogeneous for the connected Lorentz group $G=\mathrm{SO}_0(n, 1)$.

Let  $\sigma$ be the involution of $G$ defined by
conjugation with  the diagonal matrix $\diag(-1,1,\dots,1)$. The subgroup $G^\sigma$ of
$G$ of $\sigma$-fixed elements is the stabilizer of $\R x_0$ where $x_0:=(1,0,\dots,0)\in \R^{n+1}$.
One of the connected components of $G^{\sigma}$ is given by the subgroup
$$
H=\begin{pmatrix} 1&0\\0&\SO_0(n\!-\!1,1)\end{pmatrix}\subset G.
$$
Then the homogeneous space
$$
G/H
=\SO_0(n,1)/\SO_0(n-1,1)
$$
is a symmetric space. Since $G$ acts transitively on the hyperboloid
$$
X_{n}
:=\{x\in\R^{n+1}\mid x_1^2+\dots +x_n^2-x_{n+1}^2=1\},
$$
and $H$ is nothing but the stabilizer of $x_0$, it follows that $X_{n}\simeq G/H$.

Let $\Delta$ be the $G$-invariant Laplace-Beltrami operator, which is a scalar multiple of the differential operator induced by the Casimir element associated to the Lie group $G$.
For each $\lambda\in\C$ let $\Ec_\lambda(X_{n})$ be the eigenspace
$$
\Ec_\lambda(X_{n})
=\{f\in C^\infty(X_{n})\mid \Delta f=(\lambda^2-\rho^2)f\},
$$
where
$$
\rho
:=\tfrac12(n-1).
$$
For every irreducible $H$-spherical representation of $(\pi,E)$ of $G$ there exists a $\lambda\in\C$ and a $G$-equivariant injection $E\hookrightarrow \Ec_{\lambda}(X_{n})$.
In view of the $G$-equivariant symmetry $x\mapsto -x$, the space $C^\infty(X_{n})$ admits a $G$-equivariant decomposition
$$
C^\infty(X_{n})
=C_{\even}^\infty(X_{n})\oplus C_{\odd}^\infty(X_{n}).
$$
Let us write
$$
\Ec^{\even}_\lambda(X_{n})
=\Ec_\lambda(X_{n})\cap C_{\even}^\infty(X_{n}),
\quad \Ec^{\odd}_\lambda(X_{n})
=\Ec_\lambda(X_{n})\cap C_{\odd}^\infty(X_{n}).
$$
We now recall the following result proved in \cite[Theorem 1.2]{KKS}.
\begin{theorem} \label{thmkksh1}
For every $0<\lambda<\rho$ with $\lambda\in \rho-\N$ the $G$-representations $\Ec^{\even}_\lambda(X_{n})$ and $\Ec^{\odd}_\lambda(X_{n})$ are irreducible and infinitesimally equivalent. Moreover in this case,
	\begin{enumerate}
		\item
		if $\lambda-\rho$ is even then $\Ec^{\odd}_\lambda(X_{n})$ is contained in $L^2(X_{n})$
		and $\Ec^{\even}_\lambda(X_{n})$ is not $G$-tempered, and
		\item
		if $\lambda-\rho$ is odd then $\Ec^{\even}_\lambda(X_{n})$ is contained in $L^2(X_{n})$
		and $\Ec^{\odd}_\lambda(X_{n})$ is not $G$-tempered.
	\end{enumerate}
\end{theorem}

Next, we describe the $K$-types. The stabilizer of $e_{n+1}:=(0,0,\cdots, 1)\in \R^{n+1}$ is a maximal compact subgroup $K$ of $G$. Note that $K$ is isomorphic to $SO(n)$. It is convenient for our purpose to work with the coordinate system on $X_{n}$ induced by the diffeomorphism
\begin{equation}\label{eq polar coords}
S^{n-1}\times\R{\to} X_{n},\quad (y,t)\mapsto (y_1\cosh t,\dots,y_n\cosh t,\sinh t)\in X_{n}.
\end{equation}
Now as described in \cite[Section 2]{KKS}, the space $C_K^\infty(X_{n})$ of $K$-finite functions in $C^\infty(X_{n})$ given by
$$C_K^\infty(X_{n})\underset{K}{\simeq} \oplus_{j=0}^\infty \big(\Hc_j\otimes C^\infty(\R)\big)$$ where for each $j\geq 0$, $\Hc_j\subset C^\infty(S^{n-1})$ denotes  the space of spherical harmonics of degree $j$. So, the Harish-Chandra module of $\Ec_\lambda(X_{n})$ is given by
$$
\Ec_{\lambda, K}(X_{n})
=\Ec_\lambda(X_{n})\cap C^{\infty}_K(X_{n}).
$$
Consequently,
$$
\Ec_{\lambda,K}^{\even}(X_{n})=\oplus_{j=0}^\infty \,\Ec_{\lambda,j}^{\even}(X_{n}), \qquad
\Ec_{\lambda,K}^{\odd}(X_{n})=\oplus_{j=0}^\infty \,\Ec_{\lambda,j}^{\odd}(X_{n})
$$
where $\Ec_{\lambda,j}^{\even}(X_{n}) (\text{or}~ \Ec_{\lambda,j}^{\odd}(X_{n}))$ denotes the space of all even (or odd) functions in $ \big(\Hc_j\otimes C^\infty(\R)\big)$ and $\Ec_{\lambda,j}^{\even}(X_{n})\underset{K}{\simeq} \Ec_{\lambda,j}^{\odd}(X_{n})\underset{K}{\simeq} \Hc_j$ are equivalent
irreducible $K$-types for each $j$.

\par The following Proposition will be strengthened in Theorem \ref{main theorem} below. \cite[Example 5.10]{BK}

\begin{prop}
Let $n\geq 4.$  The regular representation of $G$ on $L^2(X_{n})$ is not $L^p$-factorizable for some $1<p<2.$
\end{prop}

The proposition is a special case of \cite[Example 5.10]{BK}. We give an independent proof.

\begin{proof}
Since $n\geq 4,$ clearly $\rho>1.$ Take $0<\lambda<\rho$ with $\lambda\in \rho-\mathbb{N}.$ By Theorem \ref{thmkksh1} there exists a subrepresentation $E$ of $L^2(X_{n})$ where $E=\Ec^{\odd}_\lambda(X_{n})$ or $\Ec^{\even}_\lambda(X_{n})$ depending on $\lambda-\rho$ is even or odd. From the above discussion, it is clear that $E^K\neq\{0\}.$ Now if $E$ is $G$-tempered, then we must have that all matrix coefficients $m_{v,\eta}$ with $\eta\in E^{K}\neq\{0\}$, lie weakly in $L^2(G/K).$ But then $E$ is a unitary principal series representation. The infinitesimal character of unitary principal series representations are such that $\lambda\in i\R$. This contradicts with the assumption $0<\lambda<\rho$.
It follows that $E$ is not $G$-tempered and hence the matrix coefficients are not contained in $L^{2+\varepsilon}(G)$ for suffiently small $\varepsilon>0$. The assertion now follows from Proposition \ref{Prop Factorizability of Banach reps}.
\end{proof}

\section{Smooth-$L^{p}$-factorization}
\subsection{Definition}
Let $G$ be a unimodular Lie group.
The left and right actions of $G$ on itself naturally induce a Banach representation of $G\times G$ on $L^p(G)$. The space of all smooth vectors for this representation in $L^{p}(G)$ is given by
$$
L^p(G)^{\infty}
=\{f \in C^{\infty}(G): L(u)R(v)f\in L^p(G),~~\forall u,v\in \mathcal{U}(\mathfrak{g})\},
$$
where $L$ and $R$ denote the left and right action of $U(\gf)$, respectively.
Now if  a Banach representation $(\pi, E)$ of $G$ is $L^p$-factorizable, then the $G$-equivariance of the map $L^p(G)\times E\rightarrow E$ yields that the canonical map
$$
L^p(G)^{\infty}\times E^{\infty}\rightarrow E^{\infty}
$$
is continuous. This motivates the following definition.
\begin{definition}
We say that a representation $(\pi,E)$ of $G$ is {\em{smoothly-$L^p$-factorizable}} provided that the canonical map
$$
C_c^{\infty}(G)\times E^{\infty}\rightarrow E^{\infty},\quad(f,v)\rightarrow \pi(f)v
$$
extends to a continuous bi-linear map
$$
L^p(G)^{\infty}\times E^{\infty} \rightarrow E^{\infty}.
$$
\end{definition}

In the remainder of this section we assume that $G$ is reductive and we will show that the left-regular representation of $G$ on the Harish-Chandra Schwartz space of $G$ is smoothly-$L^{p}$-factorizable for all $1\leq p<2$.

\subsection{Harish-Chandra Schwartz space}

We fix a  Cartan involution $\theta$ and denote by $K$ the maximal compact subgroup which is fixed by $\theta$.  By abuse of notation we write $\theta$ for the derived automorphism of the Lie algebra $\mathfrak{g} $ of $G$, which induces the Cartan decomposition $\mathfrak{g}=\mathfrak{k}\oplus\mathfrak{s}$. Here $\Lie(K)=\mathfrak{k}$. Let $\mathfrak{a}\subset\mathfrak{s}$ be a maximal abelian subspace. We denote the set of restricted roots of $\mathfrak{a}$ in $\mathfrak{g}$ by $\Sigma \subset \mathfrak{a}^{*}\setminus \{0\}$ and record the following root space decomposition
$$
\mathfrak{g}
=\mathfrak{a}\oplus\mathfrak{m}\oplus\left(\bigoplus_{\alpha\in \Sigma}\mathfrak{g}^{\alpha}\right).
$$
Let $A$ be the connected  subgroup of $G$ with Lie algebra $\mathfrak{a}$. Note that the usual exponential map $\exp: \mathfrak{a}\rightarrow A$ is a diffeomorphism and we write $\log :A\rightarrow \mathfrak{a}$ for its inverse. With respect to a fixed choice of Weyl chamber $\af^{+}$, we denote the system of positive roots by $\Sigma^{+}$. We then write $\mathfrak{n}:=\oplus_{\alpha\in \Sigma^+}\mathfrak{g}^{\alpha}$
and $ N :=\exp \mathfrak{n}$. Now in view of the Iwasawa decomposition $G=NAK$, we define the projection ${\bf{a}}:G\rightarrow A$ by $g\in N{\bf{a}(g)}K$ for $g\in G$. We let $\rho:= \frac12 \sum_{\alpha\in \Sigma^+}\dim (g^{\alpha})\alpha \in \mathfrak{a}^{*}$ and for any $\lambda\in \mathfrak{a}^{*}$ and $a\in A$, we define $a^{\lambda}=e^{\lambda(\log a)}$.
Let $dk$ denote the invariant probability measure on $K$ and let $da$ be a Haar measure on $A$. We normalize the Haar measure on $G$ so that with respect to the polar decomposition $G=K\overline{A^+}K$, where $A^+:= \exp(\mathfrak{a}^+)$, we have
\begin{equation}
\label{intpol}
\int_G\varphi(g)\,dg
= \int_{K}\int_{\overline{A^+}}\int_{K}\varphi(k_1ak_2) J(a)\,dk_1\,da\,dk_2
\qquad\big(\varphi\in C_{c}(G)\big)
\end{equation}
where
$$
J(a)
:=\prod_{\alpha\in \Sigma^{+}} \left(e^{\alpha(\log a)}-e^{-\alpha(\log a)}\right)^{m_{\alpha}}\qquad(a\in A).
$$

Recall the Harish-Chandra $\Xi$ function
$$
\Xi(g)
:= \int_{K}{\bf{a}}(kg)^{\rho}\,dk.
$$
This is a $K$-bi-invariant smooth function on $G$, which, in view of the polar decomposition can be realised as a function on $\overline{A^+}$. As $\Xi$ is a spherical function, it features the mean value property
\begin{equation}
	\label{mvtxi}
\int_{K}\Xi(xky)\,dk
=\Xi(x)\Xi(y)\qquad(x,y\in G).
\end{equation}
We now record the following estimate
\begin{equation}
\label{sphest}
a^{-\rho}\leq \Xi(a)\le C a^{-\rho}(1+\| \log a\|)^d\qquad \big(a\in \overline{A^+}\big)
\end{equation}
where $C$ and $d$ are constants.
We define
$$
r(g)
:=\|\log a\|\qquad \big(g=k_1ak_2\in K\overline{A^+}K\big).
$$
It is not difficult to see that
\begin{equation}\label{subadwt}
r(xy)\leq r(x)+r(y)\qquad(x,y\in G).
\end{equation}

For each $n\in \mathbb{Z}$, we let $\mathcal{C}_n(G)$ stand for the set of all continuous functions $f$ on $G$ such that
\begin{equation} \label{pn}
p_n(f)
:=\sup_{g\in G}|f(g)|\Xi(g)^{-1}(1+r(g))^n<\infty.
\end{equation}
For each $u,v \in U(\gf)$, we consider the seminorm
$$
p_{n,u,v}(f)
:= p_n\big(L(u)R(v)f\big).
$$
The Harish-Chandra Schwartz space $\mathcal{C}(G)$ is the space of all smooth functions $f$ on $G$ for which $p_{n,u,v}(f)<\infty$ for all $u,v\in U(\gf)$, and $n\in \mathbb{Z}$. It is equipped with the Fr\'echet topology generated by the semi norms $p_{n,u,v}$. It is worth noting that an equivalent characterization of $\mathcal{C}(G)$ can be obtained by replacing $p_n$ by the norms $q_{n}$ given by
$$
q_n(f)
:= \left(\int_{G}|f(g)|^2(1+r(g))^n\,dg\right)^{\frac12}.
$$
See e.g., \cite[Theorem 9 on page 348]{Var}.
Harish-Chandra proved that the convolution product defines a continuous map
$$
\mathcal{C}(G)\times \mathcal{C}(G)\to \mathcal{C}(G),\quad (f,h)\mapsto f*h.
$$
See e.g., \cite[Theorem 18 on page 357]{Var}.

\subsection{Bernstein's invariant Sobolev lemma}
Let $H\subset G$ be a closed subgroup of $G$. We consider the homogeneous space $Z:= G/H$. Assume that $Z$ is unimodular, that is, $Z$
carries a positive $G$-invariant Radon measure. We call a locally bounded function $w: Z\rightarrow \R_{>0}$ a weight on $Z$ provided for any compact subset $\Omega\subset G$, there exists $C>0$ such that
$$
w(gz)
\leq C w(z)\qquad (z\in Z, g\in \Omega).
$$
We fix a ball $B$ in $G$, by which, we always mean a compact symmetric neighborhood of $\bf{1}$ in $G$. Define
$$
{\bf{v}}(z)
=\text{vol}_{Z}(B\cdot z)\qquad (z\in Z).
$$
It is not difficult to see that $\bf{v}$ is a weight. We call $\bf{v}$ the volume weight on $Z$ corresponding to the ball $B$. Different choices of balls in $G$ give rise to equivalent weights in the sense that if $v_1$ and $v_2$ are two weights corresponding to balls $B_1$ and $B_2$, respectively, then there exists a constant $c>0$ such that
$$
\frac{1}{c} v_1(z)
\leq v_2(z)
\leq c v_1(z)\qquad (z\in Z).
$$

\begin{lemma}[Bernstein's invariant Sobolev lemma]
Let $1\leq p<\infty $ and $k\in \mathbb{N}$ be such that $pk>\dim G$. Fix a ball $B$ in $G$ and let $\bf{v}$ be the corresponding volume weight. Let $\|\cdot\|_{p,k,B\cdot z}$ be the $k$'th $L^{p}$ Sobolev norm on $B\cdot z$. Then for all $f\in C^{\infty}(Z)$, there exists constant $C>0$ such that
$$
|f(z)|
\leq C{\bf{v}}(z)^{-\frac1p}\|f\|_{p,k, B\cdot z}\qquad (z\in Z).
$$
\end{lemma}
\begin{proof} See \cite[Sect. 3.4, Key Lemma]{B} or  \cite[Lemma 4.2]{KSch}. \end{proof}

We apply the invariant Sobolev theorem to $G$ viewed as the homogeneous space $G\times G/ \diag(G)$. We choose a ball $B\subset G$ and define $\bf{v}:G\to \R_{>0}$ be the volume weight corresponding to the ball $B\times B$ of $G\times G$, i.e.,
$$
{\bf{v}}(x)
=\text{Vol}(BxB)\qquad (x\in G).
$$
In this context we record that there exists a $c>0$ so that
\begin{equation}
\label{ballest}
c^{-1}a^{2\rho}\leq {\bf{v}}(x) \leq ca^{2\rho} \qquad (x \in KaK, a\in \overline{A^+}).
\end{equation}
Hence, in view of the above lemma, we observe that for every $k>\frac{\dim (G\times G)}{p}$ there exists a $C>0$ so that
\begin{equation}
	\label{brlemmacon}
	|f(x)|\leq C {\bf{v}}(x)^{-\frac1p}\|f\|_{p,k, BxB}\leq C {\bf{v}}(x)^{-\frac1p}\|f\|_{p,k}
\end{equation}
for any $f\in L^p(G)^{\infty}$.

\begin{lemma}
Let $1\leq p<2$ and $k\in \N$ with $k>\frac{\dim (G\times G)}{p}$. For each $n\in \mathbb{Z}$ there exists a constant $C_{n}>0$ so that
$$
p_n(f\ast h)
\leq C_n\|f\|_{p,k}~~p_n(h)\qquad \big(h\in  \mathcal{C}_n(G), f\in L^p(G)^{\infty}\big).
$$
\end{lemma}

\begin{proof}
Fix $n \in \mathbb{Z}$. Using \eqref{brlemmacon}, we observe that there exists a constant $C>0$ so that for every $f\in L^{p}(G)^{\infty}$ and $h\in \mathcal{C}_{n}(G)$
\begin{align}\label{1steq}
	\nonumber
    |f\ast h(x)|&=\left|\int_G f(g)h(g^{-1}x)\,dg\right|\\
	&\leq C\|f\|_{p,k}p_n(h)\int_{G} {\bf{v}}(g)^{-\frac1p}\Xi(g^{-1}x)\left(1+r(g^{-1}x)\right)^{-n}\,dg.
\end{align}
In view of \eqref{intpol} the last integral takes the form
\begin{align*}
	\int_{K}\int_{A^+}{\bf{v}}(a)^{-\frac1p}\Xi(a^{-1}kx)\left(1+r(a^{-1}kx)\right)^{-n} J(a)\,dk\,da.
\end{align*}
In order to estimate this, we first notice that for every $k\in K$, $a\in A$ and $x\in G$
$$
\left(1+r(a^{-1}kx)\right)^{-n}\leq (1+r(x))^{-n}(1+r(a))^{|n|}.
$$
Indeed, using \eqref{subadwt}, we first see that
$$
1+r(a^{-1}kx)
\leq 1+r(a)+r(x)\leq (1+r(x))(1+r(a)).
$$
So, if $n<0$, then we have
$$
(1+r(a^{-1}kx))^{-n}
\leq (1+r(x))^{-n}(1+r(a))^{-n}.
$$
Also we observe that
$$
1+r(x)
=1+r(aa^{-1}kx)
\leq (1+r(a))(1+r(a^{-1}kx)).
$$
Now if $n>0$, the above  yields
$$
(1+r(a^{-1}kx))^{-n}
\leq {(1+r(x))^{-n}}{(1+r(a))^{n}}
$$
proving the claim. Using this estimate we have
\begin{align*}
	&\int_{K}\int_{A^+}{\bf{v}}(a)^{-\frac1p}\Xi(a^{-1}kx)\left(1+r(a^{-1}kx)\right)^{-n} J(a)\,dk\,da\\
    &\leq (1+r(x))^{-n} \int_{A^+}{\bf{v}}(a)^{-\frac1p}\left(\int_{K}\Xi(a^{-1}kx)dk\right)\left(1+r(a)\right)^{|n|} J(a)\,da\\
	&=(1+r(x))^{-n}\Xi(x) \int_{A^+}{\bf{v}}(a)^{-\frac1p}\Xi(a^{-1})\left(1+r(a)\right)^{|n|} J(a)\,da
\end{align*}
where in the last equality we have used the mean value property
\eqref{mvtxi}  of the $\Xi$-function. We now use \eqref{sphest}, and \eqref{ballest} to see that
\begin{align*}
C_n:&=\int_{A^+}{\bf{v}}(a)^{-\frac1p}\Xi(a^{-1})\left(1+r(a)\right)^{|n|} J(a)\,da\\
&\lesssim \int_{A^+}a^{-\frac2p\rho}a^{-\rho}(1+\|\log a\|)^d\left(1+r(a)\right)^{|n|} a^{2\rho}\,da\\
&= \int_{A^+} e^{-\left(\frac2p-1\right)\rho(\log a)}(1+\|\log a\|)^d\left(1+r(a)\right)^{|n|}\,da,
\end{align*} which is finite as $p<2$. From \eqref{1steq}, we obtain
$$
|f\ast h(x)|(1+r(x))^{n}\Xi(x)^{-1}
\leq CC_n \|f\|_{p,k}p_n(h)\qquad (x\in G),
$$
completing the proof.
\end{proof}

As an immediate consequence we obtain the following result.
\begin{cor}\label{cor smoothlp} For $1\leq p<2$, representations
$(L, \mathcal{C}_n(G))$ and $(L, \mathcal{C}(G))$ are smoothly $L^p$-factorizable.
\end{cor}	

\section{$L^2(X)$ is not smoothly $L^p$-factorizable}

In this section we let $n\geq 4$ and let $X=X_n$.
\begin{theorem}\label{main theorem}  $L^2(X)$ is not smoothly $L^p$-factorizable for $p > \frac{n-1}{n-2}$.
\end{theorem}

\begin{proof} We saw earlier that there exists a $K$-spherical discrete series $E\subset L^2(X)$ with spectral parameter
$0<\lambda <\rho$ subject to $\rho -\lambda\in \N$.  Recall that $\rho = \frac{1}{2}(n-1)$. Let
$E^K=\C v_K$.  For any $1\leq p<\infty$ we can define $L^p$-Schwartz spaces
$\mathcal{C}^p(G)$ simply by replacing the growth constraint $\Xi$ with $\Xi^{2\over p}$ in the definition \eqref{pn} of the Harish-Chandra Schwartz space $\mathcal{C}(G)$. Note that $\mathcal{C}(G) = \mathcal{C}^2(G)$ in our generalized framework. We denote by
$\mathcal{C}^p(G\sslash K)\subset \mathcal{C}^p(G)$ the subspace of $K$-bi-invariant functions.
Assume now that $p > \frac{n-1}{n-2}$ and that $L^2(X)$ is smoothly $L^p$-factorizable.
As $\mathcal{C}^p(G)$ is dense in $L^p(G)^\infty$ we would arrive via Dirac approximation at a contradiction
if these assumptions would yield
\begin{equation} \mathcal{C}^p(G\sslash K) *v_K = {0}.\end{equation}
To proceed we invoke the action of the Casimir element $\Omega$.
Then we have $\Omega v_K = (\lambda^2 - \rho^2)v_K$. We reformulate that in
$(\Omega - c_\lambda)v_K =0$ with $c_\lambda = \lambda^2 -\rho^2$.
Now for $f \in \mathcal{C}^p(G\sslash K)$ we have
$$ 0= f* (\Omega - c_\lambda)v_K =(\Omega - c_\lambda)f *v_K $$
We claim that $(\Omega - c_\lambda)$ is an invertible operator on $\mathcal{C}^p(G\sslash K)$ provided that $1<p<2$ is sufficiently close to $2$. In fact, the Paley-Wiener Theorem for the Fourier transform identifies $\mathcal{C}^p(G\sslash K)$ with certain holomorphic functions on the strip
$|\mathrm{Re} \lambda|<\epsilon \rho$ for $\e = \frac{2}{p} - 1$, see \cite[Th. 7.8.6]{GV}.  This implies in particular $(\Omega- c_\lambda)$ is invertible provided
$|\mathrm{Re} \lambda|> ( \frac{2}{p} - 1)\rho $. For the discrete series parameter $\lambda= \rho-1$ we obtain this condition for $p > \frac{n-1}{n-2}$ and the desired contradiction.
\end{proof}


\begin{thebibliography} {10}
	
\bibitem{BK} Y. Benoist and T. Kobayashi, {\it Temperedness of reductive homogeneous spaces}, J. Eur. Math. Soc. {\bf 17} (2015), 3015--3036.

\bibitem{B} J. Bernstein, {\it On the support of Plancherel measure}, Jour. of Geom. and Physics {\bf 5}, No. {\bf 4} (1988), 663--710.
	
\bibitem{C} M. Cowling, {\it The Kunze-Stein Phenomenon}, Ann. of Math. {\bf 107} (1978), 209--234.

\bibitem{vanDijk} G. van Dijk, {\it On a Class of Generalized Gelfand Pairs}, Math. Z. {\bf 193} (1986), 581--593.

\bibitem{Eymard} P. Eymard, {\it L'alg\`ebre de {F}ourier d'un groupe localement compact}, Bull. Soc. Math. France {\bf 92} (1964), 181--236.

\bibitem{GV} R. Gangolli and V.S. Varadarajan, {\it Harmonic Analysis of Spherical Functions on Real Reductive Groups}, Ergebnisse der Mathematik und ihrer Grenzgebiete {\bf 101}, Springer, 1988.

\bibitem{KSch} B. Kr\"otz, and H. Schlichtkrull, {\it Harmonic Analysis for Real Spherical Spaces}, Acta Mathematica Sinica, English Series   Mar., 2018, Vol. 34, No. 3, pp. 341--370

\bibitem{KKS} B. Kr{\"o}tz, J.J. Kuit, and H. Schlichtkrull, {\it Discrete series representations with non-tempered embedding}, Indag. Math. (N.S.) {\bf 33} (2022), 869--879.

\bibitem{KS} R. A. Kunze, and E. M. Stein, {\it Uniformly Bounded Representations and Harmonic Analysis of the 2 x 2 Real Unimodular Group}, Am. J. of Math. {\bf 82} (1960), 1--62.

\bibitem{SW} E. Samei and M. Wiersma, {\it Exotic {$\rm C^*$}-algebras of geometric groups}, J. Funct. Anal. {\bf 286} (2024)

\bibitem{Var}V. S. Varadarajan, {\it Harmonic analysis on real reductive groups}, Lecture Notes in Mathematics Vol. 576, Springer-Verlag, Berlin-New York, 1977.

\bibitem{W2} N. Wallach, {\it Real Reductive Groups II}, Academic Press, 1992.
\end{thebibliography}
\end{document}